\documentclass[12pt,english]{article}
\usepackage[T1]{fontenc}
\usepackage[latin9]{inputenc}
\usepackage{geometry}
\usepackage{amsmath}
\usepackage{amsthm}
\usepackage{amssymb}

\makeatletter
\numberwithin{equation}{section}
\theoremstyle{plain}
\newtheorem{thm}{\protect\theoremname}
  \theoremstyle{remark}
  \newtheorem{rem}{\protect\remarkname}
  \theoremstyle{plain}
  \newtheorem{cor}{\protect\corollaryname}
  \theoremstyle{plain}
  \newtheorem{lem}{\protect\lemmaname}
  \theoremstyle{definition}
  \newtheorem{defn}{\protect\definitionname}
 \theoremstyle{definition}
  \newtheorem{example}{\protect\examplename}

\date{}

\makeatother

\usepackage{babel}
  \providecommand{\definitionname}{Definition}
  \providecommand{\examplename}{Example}
  \providecommand{\lemmaname}{Lemma}
  \providecommand{\remarkname}{Remark}
\providecommand{\corollaryname}{Corollary}
\providecommand{\theoremname}{Theorem}

\begin{document}

\title{A Non-vanishing Property for the Signature of a Path}
\author{Horatio Boedihardjo\thanks{Department of Mathematics and Statistics, University of Reading,
Reading RG6 6AX, United Kingdom. Email: h.s.boedihardjo@reading.ac.uk.}, Xi Geng\thanks{Department of Mathematical Sciences, Carnegie Mellon University, Pittsburgh
PA 15213, United States. Email: xig@andrew.cmu.edu.}}
\maketitle

\begin{abstract}
We prove that a continuous path with finite length in a real Banach
space cannot have infinitely many zero components in its signature unless it is tree-like.
In particular, this allows us to strengthen a limit theorem for signature recently
proved by Chang, Lyons and Ni. What lies at the heart of our proof is a complexification idea together with deep results from holomorphic polynomial approximations in the theory of several complex variables.
\end{abstract}

\section{Introduction and main result}

In the seminal work of Hambly and Lyons \cite{HL10} in 2010, it was
shown that the \textit{signature} of a continuous path $w:[0,T]\rightarrow\mathbb{R}^{d}$
with finite length, which is the collection 
\[
\left\{ \int_{0<t_{1}<\cdots<t_{n}<T}dw_{t_{1}}\otimes\cdots\otimes dw_{t_{n}}:n\geqslant1\right\} 
\]
of global iterated integrals of all orders, uniquely determines the
path $w$ up to a tree-like equivalence (heuristically, a path is
tree-like if it goes out and reverses back along itself). In particular,
there is a unique \textit{tree-reduced} path (i.e. not containing
any tree-like pieces) up to reparametrization with minimal length
in each tree-like equivalence class with given signature. Since then,
it has been conjectured that the length $L$ of a tree-reduced path
$w$ can be recovered from (the tail asymptotics of) its signature:
\begin{equation}
L=\limsup_{n\rightarrow\infty}\left\Vert n!\int_{0<t_{1}<\cdots<t_{n}<T}dw_{t_{1}}\otimes\cdots\otimes dw_{t_{n}}\right\Vert _{{\rm proj}}^{\frac{1}{n}},\label{eq: length conjecture}
\end{equation}
where $\|\cdot\|_{{\rm proj}}$ is the projective tensor norm on the
tensor product. This conjecture was proved by Hambly-Lyons \cite{HL10}
for $C^{1}$-paths (a stronger asymptotic result was obtained in this
case) and piecewise linear paths, and it remains open for general
bounded variation paths. 

In a recent work of Chang, Lyons and Ni \cite{CLN18} (see also \cite{CLN18C}),
under reasonable tensor algebra norms, it was shown that the right
hand side of (\ref{eq: length conjecture}) is indeed a limit when
$n$ is taken over degrees at which the signature is nonzero. To be
precise, let $V$ be a real Banach space and $V^{\otimes_{a}n}$ ($n\geqslant1$)
be the algebraic tensor products. Recall from \cite{CLN18} that a
sequence of tensor norms $\|\cdot\|_{V^{\otimes_{a}n}}$ are call
\textit{reasonable tensor algebra norms} if\\
\\
(i) $\|\xi\otimes\eta\|_{V^{\otimes_{a}(m+n)}}\leqslant\|\xi\|_{V^{\otimes_{a}m}}\cdot\|\eta\|_{V^{\otimes_{a}n}}$
for $\xi\in V^{\otimes_{a}m}$, $\eta\in V^{\otimes_{a}n}$;
\\
(ii) $\|\phi\otimes\psi\|\leqslant\|\phi\|\cdot\|\psi\|$ for $\phi\in(V^{\otimes_{a}m})^{*}$,
$\psi\in(V^{\otimes_{a}n})^{*}$, where the norms are the induced
dual norms;
\\
(iii) $\|P^{\sigma}\xi\|_{V^{\otimes_{a}n}}=\|\xi\|_{V^{\otimes_{a}n}}$
for $\xi\in V^{\otimes_{a}n}$ and $\sigma$ being a permutation of
order $n$, where $P^{\sigma}$ is the induced permutation operator
on $n$-tensors.\\
\\
It can be shown (c.f. Diestel and Uhl \cite{DU77}) that the inequalities
in (i) and (ii) are automatically equalities. The completion of $V^{\otimes_{a}n}$
under $\|\cdot\|_{V^{\otimes_{a}n}}$ is denoted as $(V^{\otimes n},\|\cdot\|_{V^{\otimes n}})$.
Examples of reasonable tensor norms include the projective, injective
and Hilbert-Schmidt tensor norms. Throughout the rest of this article,
we will always fix a choice of reasonable tensor algebra norms.
The main result of \cite{CLN18} can be stated as follows\footnote{Indeed, in \cite{CLN18} the authors claimed the convergence as $n\rightarrow\infty$
without further restrictions. However, a careful examination of the
proof suggests that the convergence was only proved along degrees
at which the signature is nonzero. This was corrected in the corrigendum
\cite{CLN18C} of \cite{CLN18}. Theorem \ref{thm: CLN theorem} stated
above is the corrected version.}. 
\begin{thm}
\label{thm: CLN theorem}Let $w:[0,T]\rightarrow V$ be a continuous
path with finite length, and let $g=(1,g_{1},g_2,\cdots)$
be the signature of $w,$ i.e. 
\[
g_{n}\triangleq\int_{0<t_{1}<\cdots<t_{n}<T}dw_{t_{1}}\otimes\cdots\otimes dw_{t_{n}}\in V^{\otimes n},\ \ \ n\geqslant1.
\]
Define $N(g)$ to be the set of positive integers $n$ such that $g_{n}\neq0$.
Then 
\[
\lim_{\substack{n\rightarrow\infty\\
n\in N(g)
}
}\|n!g_{n}\|_{V^{\otimes n}}^{\frac{1}{n}}=\sup_{n\geqslant1}\|n!g_{n}\|_{V^{\otimes n}}^{\frac{1}{n}}.
\]
\end{thm}
\begin{rem}
The result holds for arbitrary weakly geometric rough paths, or more
generally, for any group-like elements, since the proof relies only
on the shuffle product formula (c.f. (\ref{eq: shuffle product formula}) below) of signature  which is a purely algebraic property. But with the same factorial
normalization, the result is only interesting in the bounded variation
case.
\end{rem}
On the other hand, in Theorem \ref{thm: CLN theorem}, it is a priori not clear whether
the limit can be taken over the whole integer sequence, or equivalently,
whether a continuous path with finite length can have infinitely many
zero components in its sigature. In the present article, we provide
a definite answer to this question. 
\begin{thm}
\label{thm: nonvanishing}Let $w:[0,T]\rightarrow V$ be a continuous
path with finite length in some real Banach space $V$. Then the signature
of $w$ has infinitely many zero components if and only if $w$ is
tree-like. 
\end{thm}
An immediate consequence of the above theorem is the following strengthened
version of Chang, Lyons and Ni's result.
\begin{cor}
Let $w:[0,T]\rightarrow V$ be a continuous path with finite length
in some real Banach space $V$ whose signature is $g=(1,g_{1},g_{2},\cdots)$
. Then we have 
\[
\lim_{n\rightarrow\infty}\|n!g_{n}\|_{V^{\otimes n}}^{\frac{1}{n}}=\sup_{n\geqslant1}\|n!g_{n}\|_{V^{\otimes n}}^{\frac{1}{n}}.
\]
\end{cor}
We point out that it is possible and easy to construct non-tree-like
\textit{rough} paths having vanishing signature along a subsequence
of degrees, and this makes our result non-trivial and surprising. For instance, the signature of the $2$-rough path ${\bf w}_{t}=\exp(t[{\rm e}_{1},{\rm e}_{2}])$
over $V=\mathbb{R}^{2}$ vanishes identically along odd degrees. More
generally, if $l_{t}$ is a continuous bounded variation path in the
space of degree $n$ homogeneous Lie polynomials, then the signature
of the $n$-rough path ${\bf w}_{t}=\exp(l_{t})$ vanishes identically
along degrees which are not multiples of $n$. Therefore, Theorem \ref{thm: nonvanishing} has to be a non-rough-path property, and the core of the argument, unlike the proof of Theorem \ref{thm: CLN theorem}, has to be analytic.

\section{Proof of the Main Theorem}

The sufficiency part of Theorem \ref{thm: nonvanishing} follows directly from the uniqueness result of Hambly and Lyons. For
the necessity part, our proof consists of two main ingredients.
The first one, which is a purely algebraic property, is to identify more zeros in the signature
from the given ones. The second one, which is the core of proof and relies crucially on the
bounded variation assumption, is to show that the path cannot have
``too many'' zeros in its signature unless it is tree-like. The algebraic
ingredient is relatively elementary while the analytic ingredient
relies on a complexification argument and deep results from several
complex variables.

\subsection{The algebraic ingredient}

To fix notation, for a given positive integer $d,$ denote $(d)$
as the set of positive integer multiples of $d$. The set of positive
integers is denoted as $\mathbb{Z}_{+}.$ 
\begin{lem}
\label{lem: additive subset contained in an ideal}Let $A$ be a non-empty
subset of $\mathbb{Z}_{+}$ which is closed under addition. If $\mathbb{Z}_{+}\backslash A$
contains infinitely many elements, then there exists a positive integer
$d\geqslant2,$ such that $A\subseteq(d)$.
\end{lem}
\begin{proof}
This is a direct consequence of the characterization of numerical
semigroups (c.f. Rosales and Garc$\acute{\text{i}}$a-S$\acute{\text{a}}$nchez
\cite{RG09}, Lemma 2.1). Since it is elementary, for completeness
we provide an independent proof in the appendix.
\end{proof}
Now let $g=(1,g_{1},g_{2},\cdots)$ be a tensor series, i.e. $g_{n}\in V^{\otimes n}$
for each $n$. Recall that $g$ is \textit{group-like }if it satisfies the following so-called \textit{shuffle product formula}: \textit{
\begin{equation}
g_{m}\otimes g_{n}=\sum_{\sigma\in{\cal S}(m,n)}P^{\sigma}(g_{m+n})\ \ \ \forall m,n\geqslant1,\label{eq: shuffle product formula}
\end{equation}
}where\textit{ }${\cal S}(m,n)$ is the subset of $(m,n)$-shuffles
in the permutation group of order $m+n$. It is standard that the signature
of a weakly geometric rough path (in particular, of a bounded variation
path) is group-like. By applying Lemma \ref{lem: additive subset contained in an ideal}
to the context of group-like elements, we obtain the following result
which is the algebraic ingredient for the proof of Theorem \ref{thm: nonvanishing}. 
\begin{lem}
\label{lem: vanishing lemma} Let $g$ be a group-like element. If
$g$ vanishes along a subsequence of degrees, then there exists a
positive integer $d\geqslant2$ such that $g$ vanishes identically
along degrees outside $(d)$.
\end{lem}
\begin{proof}
Let $N(g)\subseteq\mathbb{Z}_{+}$ be the set of degrees along which
$g$ vanishes. The result is trivial if $N(g)=\mathbb{Z}_{+}.$ Otherwise,
suppose that $A\triangleq\mathbb{Z}_{+}\backslash N(g)$ is non-empty.
Let $i,j\in N(g).$ Since $g_{i}$ and $g_{j}$ are both non-zero,
according to the shuffle product formula (\ref{eq: shuffle product formula}) and the reasonability of
tensor norms, we have 
\[
i!j!\|g_{i}\|_{V^{\otimes i}}\cdot\|g_{j}\|_{V^{\otimes j}}\leqslant(i+j)!\|g_{i+j}\|_{V^{\otimes (i+j)}},
\]
and thus $g_{i+j}\neq0.$ Therefore, $A$ is closed under addition.
Since $N(g)$ is an infinite set by assumption, we conclude from Lemma
\ref{lem: additive subset contained in an ideal} that $A\subseteq(d)$
for some $d\geqslant2.$ In other words, $g$ vanishes identically
along degrees outside $(d).$
\end{proof}

\subsection{The analytic ingredient}

Note that Lemma \ref{lem: vanishing lemma} relies only on the group-like
property of signatures. To complete the proof of Theorem \ref{thm: nonvanishing},
it remains to show that the signature of a bounded variation path cannot vanish identically outside $(d)$ for some $d\geqslant2$
unless it is tree-like. 

Let us first describe the underlying intuition. Suppose
that $w$ is a bounded variation path whose signature vanishes identically
outside $(d)$. If we complexify our underlying space and take $\lambda$
to be a $d$-th root of unity, then the two paths $w$ and $\lambda\cdot w$
have the same complex signature. Since these two paths are still quite
different even modulo tree-like pieces, it is reasonable to expect
that this could not happen unless $w$ itself is tree-like. However,
as we will see, this is not a simple consequence of the uniqueness
result in \cite{HL10}, and indeed there is a very subtle issue
in the complex situation which constitutes the main challenge
for this part. 

\subsubsection{The real case}

To illustrate the idea better, we first consider the case in which
no complexification is needed, i.e. when $d$ is an even integer.
In this case, the assumption implies that odd degree components of
the signature of $w$ are identically zero. Theorem \ref{thm: nonvanishing} then
follows from a simple topological lemma below and the general uniqueness
result of Boedihardjo, Geng, Lyons and Yang \cite{BGLY16} over $\mathbb{R}$.
\begin{lem}
\label{lem: simple topological fact}Let $f:V\rightarrow V$ be a
continuous bijection over a real Banach space $V$ whose only fixed point is the origin and it preserves spheres centered at the origin. Let
$w$ be a continuous path in $V$ starting at the origin. If $w$
and $f(w)$ are equal up to a reparametrization, then $w$ must be
the trivial path.
\end{lem}
\begin{proof}
Suppose on the contrary that $w$ is non-trivial. Then there exists
some $t>0$ such that $w_{t}\neq0.$ Let $\varepsilon\triangleq\|w_{t}\|_{V}$
and define 
\[
\tau\triangleq\inf\{0\leqslant s\leqslant t:\|w_{s}\|_{V}=\varepsilon\}.
\]
Note that $w|_{[0,\tau)}$ is contained in the open ball $B_{\varepsilon}.$
Since $\|f(w)_{\tau}\|_{V}=\|w_{\tau}\|_{V}=\varepsilon$ and $f(w)_{\tau}\neq w_{\tau},$
by continuity, there exists some $\delta>0,$ such that 
\[
f(w)([\tau-\delta,\tau])\cap w([0,\tau])=\emptyset.
\]
Since $f(w)$ and $w$ differ by reparametrization, we know that a
subset of $f(w)|_{[0,\tau-\delta)}$ must coincide with $w|_{[0,\tau]}$.
This is not possible since by the assumption on $f$, we know that
$f(w)|_{[0,\tau-\delta)}$ is contained in the open ball $B_{\varepsilon}$
while $w_{\tau}$ lies on the boundary. Therefore, $w$ must be trivial. 
\end{proof}

Now we can give the proof of Theorem \ref{thm: nonvanishing} when $d$ is even. Given a path $w$, its \textit{signature path} is the path defined by \[
\mathbb{W}_{t}\triangleq\left(1,w_{t},\int_{0<s_{1}<s_{2}<t}dw_{s_{1}}\otimes dw_{s_{2}},\cdots,\int_{0<s_{1}<\cdots<s_{n}<t}dw_{s_{1}}\otimes\cdots\otimes dw_{s_{n}},\cdots\right),
\] which lives in the infinite tensor algebra $T((V))\triangleq\Pi_{n=0}^\infty V^{\otimes n}$. For each $N\geqslant1$, the \textit{truncated signature path} of order $N$ is the projection of $\mathbb{W}_t$ onto the truncated tensor algebra $T^{(N)}(V)\triangleq\oplus_{n=0}^N V^{\otimes n}$ of order $N$.

\begin{proof}[Proof of Theorem \ref{thm: nonvanishing} when $d$ is even]

Suppose that $w:[0,T]\rightarrow V$ is a continuous path starting
at the origin with finite length, whose signature $g$ vanishes identically
along odd degree components. Let $\bar{w}$ be the unique tree-reduced
path (up to reparametrization) having the same signature as $w$,
i.e. the one which does not contain any tree-like pieces or equivalently
whose signature path is simple (c.f. \cite{BGLY16}, Theorem 1.1 and
Lemma 4.6). From assumption, the two paths $-\bar{w}$ and $\bar{w}$
have the same signature. According to the uniqueness theorem for signature in \cite{BGLY16},
they are equal up to tree-like equivalence. But $-\bar{w}$ is also
tree-reduced since its signature path
is also simple. Therefore, $-\bar{w}$ and $\bar{w}$ must be equal
up to reparametrization. According to Lemma \ref{lem: simple topological fact}
applied to the map $f$ defined by $f(v)=-v$$,$ we conclude that
$\bar{w}$ must be trivial and equivalently $w$ is tree-like. 

\end{proof}
\begin{rem}
In the above argument, we have not used the bounded variation property in an essential way, and the theorem holds for paths with finite $p$-variation
for $1\leqslant p<2$ without changing the proof. The non-rough-path
regularity is used in the way that if the first level path is trivial
then the signature (or equivalently, the signature path) is trivial, which
is not true for general rough paths.
\end{rem}

\subsubsection{The complex case}

Now we consider the case when $d$ is an odd integer. Unlike the other
case, it is hard to construct a real isomorphism of $V$ leaving the
signature invariant, and the simplest way to have such invariance
is multiplying by a $d$-th root of unity which is now a complex number.
In this way, we need to complexify the underlying space and signature
needs to be understood in the complex sense. A crucial point one needs
to be aware of is that the complex signature is defined through iterated
integrals with respect to the complex variables only but \textit{not} with
their conjugates.

This will lead to a substantial challenge in the complex case to make
the previous real argument work. Indeed, the real uniqueness result for
signature does not hold over $\mathbb{C}$! Heuristically, being tree-like
is a real property and there exists non-tree-like paths with trivial
complex signature. For instance, according to Cauchy's
theorem, any simple and closed path with finite length living inside a one dimensional complex subspace
of $\mathbb{C}^{n}$ has trivial complex signature while it needs
not be tree-like. Therefore, the complex version of the uniqueness
result requires major modification, which at this point is unclear
and unknown. However, for our particular problem, we can still obtain
inspirations from the main strategy in the proof of the real uniqueness
result. Our argument in this part relies on ideas developed in \cite{BGLY16}
and deep results from several complex variables.

To start with, we first introduce some standard notions about complexification
of real Banach spaces. Recall that the complexification of $V$ is
defined by $V_{\mathbb{C}}\triangleq V\otimes_{\mathbb{R}}\mathbb{C},$
which is isomorphic to $V\oplus V$ equipped with a complex scalar
multiplication in the canonical way. A natural choice of norms on
$V_{\mathbb{C}},$ known as the \textit{Taylor complexification
norm}, is defined by 
\[
\|x+iy\|_{T}\triangleq\sup_{0\leqslant t\leqslant2\pi}\|x\cos t-y\sin t\|_{V},\ \ \ x,y\in V.
\]
The Taylor complexification norm satisfies 
\[
\|x+iy\|_{T}=\|x-iy\|_{T},\ \|x\|_{T}=\|x\|_{V},\ \ \ x,y\in V.
\]
We always endow $V_{\mathbb{C}}$ with this norm and the (complex)
tensor products $V_{\mathbb{C}}^{\otimes n}$ with the injective tensor
norm. Let $j_{n}:V^{\otimes_{a}n}\rightarrow V_{\mathbb{C}}^{\otimes_{a}n}\cong\left(V^{\otimes_{a}n}\right)_{\mathbb{C}}$
be the canonical embedding.
\begin{lem}
\label{lem: cty of j_n}For each $n\geqslant1,$ $j_{n}$ is continuous
with norm at most one, and thus extends continuously to the completion
of the algebraic tensor product.
\end{lem}
\begin{proof}
According to van Zyl \cite{Zyl08}, Theorem 2.3, the injective tensor
norm on $V_{\mathbb{C}}^{\otimes_{a}n}$ coincides with the Taylor
complexification norm induced from the injective norm on $V^{\otimes_{a}n}$.
In addition, it is known (c.f. \cite{DU77}, Chapter 8, Proposition
3) that injective tensor norm is the smallest among all reasonable
tensor norms. Therefore, for any $\xi\in V^{\otimes_{a}n},$ 
\[
\|j_{n}(\xi)\|_{V_{\mathbb{C}}^{\otimes_{a}n}}=\|j_{n}(\xi)\|_{T}=\|\xi\|_{{\rm inj}}\leqslant\|\xi\|_{V^{\otimes_{a}n}}.
\]
\end{proof}
\begin{lem}
\label{lem: consistency of signature}Let $w:[0,T]\rightarrow V$
be a continuous path with finite length. For each $n\geqslant1,$
let $g_{n}$ (respectively, $g_{n}^{\mathbb{C}}$) be the $n$-th
degree component of its signature when $w$ is viewed as a path in
$V$ (respectively, in $V_{\mathbb{C}}$). Then $g_{n}^{\mathbb{C}}=j_{n}(g_{n}).$
\end{lem}
\begin{proof}
Let $\xi_{m}$ be the discrete Riemann sum approximation of $g_{n}.$
Then $j_{n}(\xi_{m})$ is the discrete Riemann sum of $g_{n}^{\mathbb{C}}.$
The result then follows from the continuity of $j_{n}$ stated in
Lemma \ref{lem: cty of j_n}.
\end{proof}
\begin{rem}
Lemma \ref{lem: consistency of signature} remains true for any arbitrary
rough path and its complexification. In addition, we only need to be careful
about complexification of norms in the infinite dimensional setting as the
finite dimensional case is trivial in terms of norm comparison.
\end{rem}
It is easy to see that the notion of group-like property carries through
to the complex case, and the complex signature of a weakly geometric
complex rough path (in particular, of a complex bounded variation
path) is group-like.

From now on, we fix $d\geqslant3$ to be an odd integer and $\lambda\triangleq{\rm e}^{2\pi i/d}$
to be a $d$-th root of unity. To prove Theorem \ref{thm: nonvanishing}
in this case, let $w:[0,T]\rightarrow V$ be a continuous path with
finite length starting at the origin whose signature vanishes identically
outside $(d).$ We assume on the contrary that $w$ is not tree-like
(equivalently it has non-trivial signature) and look for a contradiction.
As before in the real case, we may assume without loss of generality
that $w$ is tree-reduced. It is apparent from assumption that the
path $z_{t}\triangleq\lambda\cdot w_{t}$ has the same \textit{complex}
signature as $w_{t}$. The main difficulty here is that $z$ has different
\textit{real} signature from $w$ (when regarding $V_{\mathbb{C}}=V\oplus V$
as a real vector space) so that the real uniqueness result does not
apply. 

To explain the underlying idea, assume for the moment that $\dim V<\infty$
and $w$ is a simple and closed path. If $w$ is non-trivial, it is not hard
to construct a \textit{real} continuous one form $\phi=\sum_{j}\phi_{j}(x)dx^{j}$
\textit{over} $V$ supported inside some small neighbourhood $B\subseteq V$
of $q\in\mathrm{Im}(w)\backslash\{0\}$, such that $\int_{0}^{T}\phi(dw)=1.$
Since over $V_\mathbb{C}$, $B$ is entirely separated from $\mathrm{Im}(z)$,
by zero extension each $\phi_{j}$ extends to a continuous
function $\bar{\phi}_{j}$ over the compact subset $K\triangleq{\rm Im}(w)\cup{\rm Im}(z)\subseteq V_{\mathbb{C}},$
and therefore $\phi$ extends to a continuous one form $\bar{\phi}=\sum_{j}\bar{\phi}_{j}(z)dz^{j}$
(not containing $d\bar{z}^{j}$!) over $V_{\mathbb{C}}$. In particular, from the construction
we see that 
\[
\int_{0}^{T}\bar{\phi}(dw_{t})=1\ {\rm and}\ \int_{0}^{T}\bar{\phi}(dz_{t})=0.
\]
The key to reaching a contradiction is the possibility of
approximating $\bar{\phi}$ by \textit{holomorphic} polynomial one forms (i.e.
polynomial in the complex variables but not in their conjugates).
This turns out to be a rather deep problem in the theory of
holomorphic polynomial approximations in several complex variables, and it can only
be achieved in some very special situations (fortunately, our situation
is special enough). Once we are able to replace $\bar{\phi}$ by a complex
polynomial one form $p$, a contradiction is then immediate since
the integral depends only on $p$ and the complex signature according
to the shuffle product formula.
If $V$ is infinite dimensional and $w$ is a general tree-reduced
path, one needs to work with truncated signature paths and apply finite
dimensional reduction in a proper way similar to the strategy developed in \cite{BGLY16}
when proving the real uniqueness result.

We first review a few results on holomorohic polynomial approximations in
several complex variables that will be needed for our problem.
\begin{defn}
The \textit{polynomial convex hull} of a compact subset $K\subseteq\mathbb{C}^{n}$
is defined as 
\[
\hat{K}\triangleq\left\{ z\in\mathbb{C}^{n}:|p(z)|\leqslant\sup_{w\in K}|p(w)|\ {\rm for\ all\ complex\ polynomials}\ p\right\} .
\]
A compact subset $K$ is called \textit{polynomially convex} if $\hat{K}=K$. 
\end{defn}
Polynomial convexity is closely related to uniform polynomial approximations,
as seen from the following result.

\begin{thm}[c.f. Levenberg \cite{Levenberg06}, Page 97, Corollary]\label{thm: polynomial approximation}

Let $K$ be a polynomially convex compact subset of $\mathbb{C}^{n}$
with zero $2$-dimensional Hausdorff measure. Then every continuous
function over $K$ can be uniformly approximated by holomorphic polynomials.

\end{thm}
\begin{example}
(1) Every compact subset $K$ of $\mathbb{R}^{n}$, where $\mathbb{R}^{n}$
is viewed as the real part of $\mathbb{C}^{n}$, is polynomially convex.
Therefore, any continuous function on the image of a bounded variation
path in $\mathbb{R}^{n}$ can be uniformly approximated by complex
polynomials. This can also be seen by applying the more standard real
polynomial approximation theorems (e.g. Berstein's theorem) and regard
a real polynomial as a complex polynomial in the natural way.

(2) Let $K$ be the unit circle in $\mathbb{C}^{1}.$ Then the polynomial
convex hull of $K$ is the unit disk. This partly explains why not
every continuous function on $S^{1}$ can be uniformly approximated
by holomorphic polynomials which is consistent with Cauchy's theorem.
\end{example}
The following result, which gives a way of verifying polynomial convexity
in some special situations, is crucial for us.

\begin{thm}[c.f. Weinstock \cite{Weinstock88}, Theorem 1]\label{thm: W-thm}

Let $A$ be a real $n\times n$ matrix which does not have purely
imaginary eigenvalues of modulus greater than one. Define $M\triangleq(A+i)\mathbb{R}^{n}\subseteq\mathbb{C}^{n}.$
Then every compact subset of $M\cup\mathbb{R}^{n}\subseteq\mathbb{C}^{n}$
is polynomially convex. 

\end{thm}

We now return to our signature problem. Recall that $d\geqslant3$ is
an odd integer and $\lambda={\rm e}^{i\theta}$ ($\theta\triangleq2\pi/d$)
is a $d$-th root of unity. For each $N\geqslant1,$ define the real
and complex spaces 
\[
\tilde{T}_{{\rm inj}}^{(N)}(V)\triangleq\bigoplus_{\substack{k=1\\
d\nmid k
}
}^{N}V^{\otimes_{{\rm inj}}k},\ \tilde{T}^{(N)}(V_{\mathbb{C}})\triangleq\bigoplus_{\substack{k=1\\
d\nmid k
}
}^{N}V_{\mathbb{C}}^{\otimes k},
\]
respectively, where ``inj'' means injective tensor norm. Note that
$\tilde{T}_{{\rm inj}}^{(N)}(V)$ is canonically embedded inside $\tilde{T}^{(N)}(V_{\mathbb{C}})$
as its real part. Define the dilation operator $\delta_{\lambda}:\tilde{T}^{(N)}(V_{\mathbb{C}})\rightarrow\tilde{T}^{(N)}(V_{\mathbb{C}})$
in the usual way by $V_{\mathbb{C}}^{\otimes k}\ni g_{k}\mapsto\lambda^{k}g_{k}.$
As will be clear soon, the reason we work in the space $\tilde{T}^{(N)}(V_{\mathbb{C}})$
instead of the more traditional truncated tensor algebra $T^{(N)}(V_{\mathbb{C}})$ is for technical convenience
of applying Theorem \ref{thm: W-thm} (c.f. Lemma \ref{lem: p-convex lemma} below). 
\begin{lem}
\label{lem: linear independence lemma}Let $E$ be a real vector space
and let $\left\{ v_{1}^{j}+\cdots+v_{r_{j}}^{j}:1\leqslant j\leqslant n\right\} $
be a linearly independent set. Then one can choose $v_{l_{j}}^{j}$
($1\leqslant l_{j}\leqslant r_{j}$) for each $j,$ such that $\{v_{l_{j}}^{j}:1\leqslant j\leqslant n\}$
are linearly independent.
\end{lem}
\begin{proof}
We write $v^{j}\triangleq v_{1}^{j}+\cdots+v_{r_{j}}^{j}.$ Then there exists
at least one $v_{l_{1}}^{1}$ such that $\{v_{l_{1}}^{1},v^{2},\cdots,v^{n}\}$
are linearly independent, for otherwise $v^{1}$ will be linearly dependent
on $\{v^{2},\cdots,v^{n}\}$ which is a contradiction. Similarly,
there exists at least one $v_{l_{2}}^{2}$ such that $\{v_{l_{1}}^{1},v_{l_{2}}^{2},v^{3},\cdots,v^{n}\}$
are linearly independent. Now one can proceed by induction.
\end{proof}
\begin{lem}
\label{lem: p-convex lemma}Let $L:\tilde{T}_{{\rm inj}}^{(N)}(V)\rightarrow\mathbb{R}^{n}$
be a real surjective continuous linear map, and extend $L$ to a complex continous
linear map $\bar{L}:\tilde{T}^{(N)}(V_{\mathbb{C}})\rightarrow\mathbb{C}^{n}$
in the canonical way by 
\[
\bar{L}(u+iv)\triangleq L(u)+iL(v).
\]
Let $M\triangleq\delta_{\lambda}(\tilde{T}_{{\rm inj}}^{(N)}(V))\subseteq\tilde{T}^{(N)}(V_{\mathbb{C}}).$
Then every compact subset of $\bar{L}(M)\cup\mathbb{R}^{n}$ is polynomially
convex. 
\end{lem}
\begin{proof}
Since $L$ is surjective, there exist elements $g_{1},\cdots,g_{n}\in\tilde{T}_{{\rm inj}}^{(N)}(V)$
such that $\{L(g_{1}),\cdots,L(g_{n})\}$ form a basis of $\mathbb{R}^{n}.$
Since each $g_{j}$ is a sum of homogeneous tensors, according to
Lemma \ref{lem: linear independence lemma}, we may choose some $\xi_{l_{j}}\in V^{\otimes_{{\rm inj}}l_{j}}$
$(1\leqslant j\leqslant n)$ such that $\{L(\xi_{l_{1}}),\cdots,L(\xi_{l_{n}})\}$
form a basis of $\mathbb{R}^{n}$. In addition, observe that 
\begin{align*}
M & ={\rm Span}_{\mathbb{R}}\left\{ \lambda^{k}\cdot\xi_{k}:1\leqslant k\leqslant N,d\nmid k,\ \xi_{k}\in V^{\otimes_{{\rm inj}}k}\right\} \\
 & ={\rm {\rm Span}}_{\mathbb{R}}\left\{ (\cos k\theta+i\sin k\theta)\cdot\xi_{k}:1\leqslant k\leqslant N,d\nmid k,\ \xi_{k}\in V^{\otimes_{{\rm inj}}k}\right\} .
\end{align*}
Since $d$ is odd, we know that $\sin k\theta\neq0$ for all $k$
not being a multiple of $d.$ Therefore, 
\[
M={\rm Span}_{\mathbb{R}}\left\{ (\cot k\theta+i)\cdot\xi_{k}:1\leqslant k\leqslant N,d\nmid k,\ \xi_{k}\in V^{\otimes_{{\rm inj}}k}\right\} .
\]
It follows that $\bar{L}(M)$ is an $n$-dimensional real subspace
of $\mathbb{C}^{n}$ with basis $\{(\cot l_{j}\theta+i)\cdot L(\xi_{l_{j}}):1\leqslant j\leqslant n\}.$
In particular, if we define a non-degenerate real linear transform
$A:\mathbb{R}^{n}\rightarrow\mathbb{R}^{n}$ by $A(L(\xi_{l_{j}}))\triangleq\cot l_{j}\theta\cdot L(\xi_{l_{j}}),$
then $A$ does not have purely imaginary eigenvalues (indeed, all
eigenvalues of $A$ are given by $\{\cot l_{j}\theta:1\leqslant j\leqslant n\}$),
and $\bar{L}(M)=(A+i)\mathbb{R}^{n}$. According to Theorem \ref{thm: W-thm},
we conclude that every compact subset of $\bar{L}(M)\cup\mathbb{R}^{n}$
is polynomially convex.
\end{proof}
Now we are in a position to give the proof of Theorem \ref{thm: nonvanishing} when $d$ is odd.
Recall that $w:[0,T]\rightarrow V$ is a non-trivial tree-reduced
path with finite length starting at the origin, whose signature vanishes
identically along degrees outside $(d).$ Viewed as paths in $V_{\mathbb{C}},$
we know that $w$ and $z\triangleq\lambda\cdot w$ have the same complex
signature. 

\begin{proof}[Proof of Theorem \ref{thm: nonvanishing} when $d$ is odd]

First of all, since $w$ is non-trivial, let $I\subseteq(0,T)$ be
a compact interval such that $0\notin w(I).$ Let $\mathbb{W}$ and
$\mathbb{Z}$ be the complex signature paths of $w$ and $z$ respectively, which 
live in the infinite complex tensor algebra $T((V_{\mathbb{C}}))\triangleq \Pi_{n=0}^\infty V_\mathbb{C}^{\otimes n}$.
Since $w$ is tree-reduced, we know that $\mathbb{W}$ and $\mathbb{Z}$
are both simple. Also they have the same starting and end points respectively.
Observe that 
\begin{equation}
\mathbb{W}(I)\cap\mathbb{Z}([0,T])=\emptyset,\label{eq: 0th separation}
\end{equation}
for otherwise if $\mathbb{W}_{t}=\mathbb{Z}_{s}=\delta_{\lambda}(\mathbb{W}_{s})$
for some $t\in I$ and $s\in[0,T]$, then $0\neq w_{t}=\lambda w_{s}$
which is absurd. Fix four points $s<s'<t'<t$ in $I.$ It follows
that 
\begin{equation}
\mathbb{W}([s,s'])\cap\mathbb{W}([t',t])=\emptyset\label{eq: first separation}
\end{equation}
and 
\begin{equation}
\mathbb{W}([s',t'])\cap\left(\mathbb{W}([0,s]\cup[t,T])\right)=\emptyset.\label{eq: second separation}
\end{equation}
In addition, from the triangle inequality we know that \[
\left\|\int_{u<t_{1}<\cdots<t_{n}<v}dw_{t_{1}}\otimes\cdots\otimes dw_{t_{n}}\right\|_{V_{\mathbb{C}}^{\otimes n}}\leqslant\frac{\|w\|_{1-{\rm var};[u,v]}}{n!}, \ \ \ \forall n\geqslant1\ \mathrm{and}\ u\leqslant v,
\]and the same is true for the path $z_t$.
Therefore, when $N$ is large, all of the separation
properties (\ref{eq: 0th separation}), (\ref{eq: first separation})
and (\ref{eq: second separation}) are preserved if we consider the
complex truncated signature paths $W^{N}$ and $Z^{N}$ in $T^{(N)}(V_{\mathbb{C}})\triangleq\oplus_{n=0}^N V_\mathbb{C}^{\otimes n}$. Choose $N=dm+1$ with
some large $m$ for this purpose.

We claim that the projections of $W^{N}$ and $Z^{N}$ onto $\tilde{T}^{(N)}(V_{\mathbb{C}})$,
denoted as $\tilde{W}^{N}$ and $\tilde{Z}^{N}$ respectively, preserve
all the previous three separation properties. We only verify (\ref{eq: first separation})
as the other cases can be treated in the same way. Let $u\in[s,s']$
and $v\in[t',t].$ Suppose on the contrary that $\tilde{W}_{u}^{N}=\tilde{W}_{v}^{N}.$
Since $W_{u}^{N}\neq W_{v}^{N},$ we conclude that $\pi_{pd}(W_{u}^{N})\neq\pi_{pd}(W_{v}^{N})$
for some $1\leqslant p\leqslant m,$ where $\pi_{pd}$ denotes the projection onto the $pd$-th component. But we know that $\pi_{1}(W_{u}^{N})=\pi_{1}(W_{v}^{N})\neq0$ since $\tilde{W}^N_u=\tilde{W}^N_v$ and $u,v\in I$. Therefore, using the cross-norm property we have
\begin{align*}
 & \left\Vert \pi_{pd}(W_{u}^{N})\otimes\pi_{1}(W_{u}^{N})-\pi_{pd}(W_{v}^{N})\otimes\pi_{1}(W_{v}^{N})\right\Vert _{V_{\mathbb{C}}^{\otimes(pd+1)}}\\
 & =\left\|\left(\pi_{pd}(W_{u}^{N})-\pi_{pd}(W_{v}^{N})\right)\otimes\pi_{1}(W_{u}^{N})\right\|_{V_{\mathbb{C}}^{\otimes(pd+1)}}\\
 & =\left\|\pi_{pd}(W_{u}^{N})-\pi_{pd}(W_{v}^{N})\right\|_{V_{\mathbb{C}}^{\otimes pd}}\cdot\left\|\pi_{1}(W_{u}^{N})\right\|_{V_\mathbb{C}}\\
 & \neq0,
\end{align*}which implies that
\[
\pi_{pd}(W_{u}^{N})\otimes\pi_{1}(W_{u}^{N})\neq\pi_{pd}(W_{v}^{N})\otimes\pi_{1}(W_{v}^{N}).
\]
According to the shuffle product formula, we conclude that 
\[
\pi_{pd+1}(W_{u}^{N})\neq\pi_{pd+1}(W_{v}^{N}),
\]
which is a contradiction to the assumption $\tilde{W}_{u}^{N}=\tilde{W}_{v}^{N}$ since $d\nmid pd+1$. Therefore,
(\ref{eq: first separation}) holds for the path $\tilde{W}^{N}.$ 

Next, since $\tilde{W}^{N}$ lives in $\tilde{T}_{{\rm inj}}^{(N)}(V)$
(the real part of $\tilde{T}^{(N)}(V_{\mathbb{C}})$), according to
the Hahn-Banach theorem (c.f. \cite{BGLY16}, Lemma 4.5), there exists
a real continuous linear map $L_{1}:\tilde{T}_{{\rm inj}}^{(N)}(V)\rightarrow\mathbb{R}^{n_{1}}$
with some $n_{1},$ such that 
\begin{equation}
L_{1}\left(\tilde{W}^{N}([s,s'])\right)\cap L_{1}\left(\tilde{W}^{N}([t',t])\right)=\emptyset\label{eq: projected separation 1}
\end{equation}
and 
\begin{equation}
L_{1}\left(\tilde{W}^{N}([s',t'])\right)\cap L_{1}\left(\tilde{W}^{N}([0,s]\cup[t,T])\right)=\emptyset.\label{eq: projected separation 2}
\end{equation}
Also for the same reason, there exists a real continuous linear map $f:V\rightarrow\mathbb{R}^{n_{2}}$
with some $n_{2},$ such that $0\notin f(w(I)).$ By taking images
we may assume that $L_{1}$ and $f$ are both surjective. Set $L_{2}\triangleq f\circ\pi_{1}:\tilde{T}_{{\rm inj}}^{(N)}(V)\rightarrow\mathbb{R}^{n_{2}}$
where $\pi_{1}$ is the canonical projection onto the first degree
component. With $n\triangleq n_{1}+n_{2}$, define
\[
L\triangleq L_{1}\oplus L_{2}:\tilde{T}_{{\rm inj}}^{(N)}(V)\rightarrow\mathbb{R}^{n}\cong\mathbb{R}^{n_{1}}\oplus\mathbb{R}^{n_{2}}
\]
and extend $L$ to a complex continuous linear map $\bar{L}:\tilde{T}^{(N)}(V_{\mathbb{C}})\rightarrow\mathbb{C}^{n}$
in the canonical way. It is apparent that the separation properties
(\ref{eq: projected separation 1}) and (\ref{eq: projected separation 2})
are still true in the space $\mathbb{R}^{n}$ with $L_{1}$ replaced
by $L$. Moreover, we claim that in the space $\mathbb{C}^{n}$, we also have 
\begin{equation}
\bar{L}\left(\tilde{W}^{N}(I)\right)\cap\bar{L}\left(\tilde{Z}^{N}([0,T])\right)=\emptyset.\label{eq: projected W/Z separation}
\end{equation}
Indeed, suppose on the contrary that $\bar{L}(\tilde{W}_{t}^{N})=\bar{L}(\tilde{Z}_{s}^{N})$
for some $t\in I$ and $s\in[0,T].$ By looking at the $L_{2}$-component,
we see that 
\[
f(w_{t})=\cos\theta\cdot f(w_{s})+i\sin\theta\cdot f(w_{s}).
\]
This implies that $f(w_{s})=0$ and thus $f(w_{t})=0,$ which is a
contradiction to the construction of $f.$ 

Now take four open subsets $U_{1}$, $U_{2},$ $V_{1},$
$V_{2}$ of $\mathbb{R}^{n}$, such that 
\[
L\left(\tilde{W}^{N}([s,s'])\right)\subseteq U_{1},\ L\left(\tilde{W}^{N}([t',t])\right)\subseteq U_{2},
\]
\[
L\left(\tilde{W}^{N}([s',t'])\right)\subseteq V_{1},\ L\left(\tilde{W}^{N}([0,s]\cup[t,T])\right)\subseteq V_{2},
\]
and 
\[
U_{1}\cap U_{2}=V_{1}\cap V_{2}=\emptyset.
\]
Define $F,G\in C_{c}^{\infty}(\mathbb{R}^{n})$ to be such that 
\[
F=0\ {\rm on}\ V_{2},\ F=1\ {\rm on}\ V_{1},
\]
and 
\[
G=0\ {\rm on}\ U_{1},\ G=1\ {\rm on}\ U_{2}.
\]
Consider the smooth one form $\Phi\triangleq FdG$ over $\mathbb{R}^{n}.$ From the construction, we have
\begin{align*}
 & \int_{0}^{T}\Phi(d(L\tilde{W}^{N})_{u})\\
 & =\left(\int_{0}^{s}+\int_{s}^{s'}+\int_{s'}^{t'}+\int_{t'}^{t}+\int_{t}^{T}\right)\Phi(d(L\tilde{W}^{N})_{u})\\
 & =0+0+\left(G(\tilde{W}_{t'}^{N})-G(\tilde{W}_{s'}^{N})\right)+0+0\\
 & =1-0\\
 & =1.
\end{align*}
Now regard $\Phi=\sum_{j=1}^{n}\Phi_{j}(x)dx^{j}$ as a real continuous
one form over ${\rm Im}(L\tilde{W}^{N})\subseteq\mathbb{R}^{n}$ and
extend it to a continuous one form $\bar{\Phi}\triangleq\sum_{j=1}^{n}\Phi_{j}(z)dz^{j}$
over the compact set 
\[
K\triangleq\bar{L}\left(\tilde{W}^{N}([0,T])\right)\cup\bar{L}\left(\tilde{Z}^{N}([0,T])\right)\subseteq\mathbb{C}^{n}
\]
by zero extension. This is legal because of the separation property
(\ref{eq: projected W/Z separation}) and by construction $\Phi=0$
on $L(\tilde{W}^{N}([0,s]\cup[t,T])).$ It follows that 
\begin{equation}
\int_{0}^{T}\bar{\Phi}(d(\bar{L}\tilde{W}^{N})_{u})=1\ {\rm and}\ \int_{0}^{T}\bar{\Phi}(d(\bar{L}\tilde{Z}^{N})_{u})=0.\label{eq: integral separation}
\end{equation}

From Lemma \ref{lem: p-convex lemma}, we know that $K$ is polynomially
convex. In addition, since $K$ is the union of images of bounded
variation paths, it has zero $2$-dimensional Hausdorff measure. According
to Theorem \ref{thm: polynomial approximation}, we know that $\bar{\Phi}$
can be uniformly approximated over $K$ by holomorphic polynomial one
forms in $\mathbb{C}^{n}.$ In particular, it follows from (\ref{eq: integral separation})
that there exists a holomorphic polynomial one form $P=\sum_{j=1}^{n}P_{j}(z)dz^{j}$, such that 
\begin{equation}
\int_{0}^{T}P\left(d(\bar{L}\tilde{W}^{N})_{u}\right)\neq\int_{0}^{T}P\left(d(\bar{L}\tilde{Z}^{N})_{u}\right).\label{eq: separation by polynomial}
\end{equation}

On the other hand, it is not hard to see from the shuffle product
formula that the complex signature
of $\bar{L}\tilde{W}^{N}$ as a bounded variation path over $\mathbb{C}^{n}$
is a function of the complex signature of $w$ as a bounded
variation path over $V_{\mathbb{C}}$ (c.f. \cite{BGLY16}, Lemma 4.2 and Lemma 4.3 for the more general rough path case). The same is true for $\bar{L}\tilde{Z}^{N}.$
Since $w$ and $z$ have the same complex signature, we conclude that  $\bar{L}\tilde{W}^{N}$
and $\bar{L}\tilde{Z}^{N}$ have the same complex signature. But this leads to a contradiction with
(\ref{eq: separation by polynomial}) since the integral is a function of the complex signature according to the shuffle product
formula again. 

Therefore,  the path $w$ has to be tree-like and the proof of Theorem
\ref{thm: nonvanishing} in the case when $d$ is odd is complete.

\end{proof}
\begin{rem}
The above separation property by holomorphic polynomial one forms relies crucially
on the feature that we are having a real path $w$ and its complex
rotation $z=\lambda\cdot w$ (or more precisely, a real path $\tilde{W}^{N}$
and its complex dilation $\tilde{Z}^{N}=\delta_{\lambda}(\tilde{W}^{N})$).
A similar separation property for two general complex paths is highly
non-trivial, and to our best knowledge this question is not fully understood in the literature.
We expect that a proper understanding on this question will be an essential
ingredient if one wants to investigate the uniqueness problem for signature over the complex field.
\end{rem}

\section*{Acknowledgement}

The authors would like to thank Professor Terry Lyons for his helpful discussions and advices. The authors would also like to thank the referee for his/her helpful suggestions, and in particular a very insightful comment which leads to significant enhancement of the analysis developed in the complexification argument. The first author is supported by EPSRC grant EP/R008205/1. The second author is supported by NSF grant DMS1814147.

\section*{Appendix}

In this section, for completeness we give an independent proof of Lemma \ref{lem: additive subset contained in an ideal}.

\begin{proof}[Proof of Lemma \ref{lem: additive subset contained in an ideal}]

Fix some $n\in A$. Apparently $n\geqslant2,$
otherwise by the additivity assumption, we have $A=\mathbb{Z}_{+}$
which is a contradiction. 

For each $i\in\mathbb{Z}_{n}$ (the integer group modulo $n$), denote
\[
[i]\triangleq\{i,i+n,i+2n,i+3n,\cdots\}.
\]
Let $G\subseteq\mathbb{Z}_{n}$ be the collection of elements $i$
such that $[i]\cap A\neq\emptyset$. We claim that $G$ is a subgroup
of $\mathbb{Z}_{n}.$ Indeed, for $i,j\in G,$ if both of $i+kn$
and $j+ln$ belong to $A$ for some $k$ and $l,$ by assumption we
see that $i+j+(k+l)n$ belongs to $A$. Thus $G$ is closed under
addition. Moreover, by the same reason, the inverse of $i\in G$,
which is the congruence class of $(n-1)i,$ also belongs to $G$.
Therefore, $G$ is a subgroup of $\mathbb{Z}_{n}.$

Note that $\mathbb{N}=\cup_{i\in\mathbb{Z}_{n}}[i]$ and $\mathbb{N}\backslash A$
is an infinite set, so there must exist some $i_{0}\in\mathbb{Z}_{n}$
such that $[i_{0}]\cap(\mathbb{N}\backslash A)$ is an infinite set.
However, since $n\in A,$ if $i_{0}+kn\in\mathbb{N}\backslash A,$
by the additivity assumption we see that $i_{0}+(k-1)n\in\mathbb{N}\backslash A$.
Therefore, we conclude that $[i_{0}]\subseteq\mathbb{N}\backslash A$
and thus $G$ is a proper subrgoup of $\mathbb{Z}_{n}.$ 

Since $\mathbb{Z}_{n}$ is cyclic, as a subgroup $G$ must also be cyclic. Let $d\triangleq\min\{i:i\in G\}\geqslant2.$
Then $d$ is a generator of $G$ in $\mathbb{Z}_{n}$ and $d$ is
a common divisor of all elements in $G.$ Note that $d$ divides $n$
by Language's theorem. Therefore, we conclude that
\[
A\subseteq\bigcup_{i\in G}[i]\subseteq(d).
\]

\end{proof}

\end{document}